\documentclass[11pt]{article}
\usepackage{graphicx} % Required for inserting images
\usepackage{amsmath, amsthm}
\usepackage{color}
\usepackage[latin1]{inputenc}
\usepackage{enumerate}
\usepackage[hang,flushmargin]{footmisc}
\usepackage{amsmath}
\usepackage{amsfonts}
\usepackage{amssymb}
\usepackage{amsthm}
\usepackage{xcolor}
\usepackage{dirtytalk}
\usepackage{cleveref}
\usepackage{thmtools, thm-restate} % For restating props and thms

\usepackage[justification=raggedright]{caption}

\theoremstyle{definition} 
\newtheorem{theorem}{Theorem} 
  
\newtheorem{lemma}[theorem]{Lemma}

\newtheorem{conjecture}[theorem]{Conjecture}     
\newtheorem{example}[theorem]{Example}        
\newtheorem{remark}[theorem]{Remark}
\newtheorem{question}[theorem]{Question}

\newcommand*\samethanks[1][\value{footnote}]{\footnotemark[#1]} % For affiliation of authors

\title{Embedding dimension gaps in sparse codes}
\author{R. Amzi Jeffs\thanks{Department of Mathematical Sciences, Carnegie Mellon University, Pittsburgh, PA 15213, USA} \thanks{Supported by the National Science Foundation through Award No. 2103206.} \and Henry Siegel\samethanks[1] \and David Staudinger\samethanks[1] \and Yiqing Wang\samethanks[1]}

\date{September 2023}

\DeclareMathOperator{\conv}{conv} % Convex hull of a set
\DeclareMathOperator{\code}{code} % Code of a collection of sets
\DeclareMathOperator{\odim}{odim} % Open embedding dimension
\DeclareMathOperator{\cdim}{cdim} % Closed embedding dimension
\DeclareMathOperator{\CF}{CF} % Canonical form of a code
\newcommand{\R}{\mathbb{R}} % macro for R^n
\newcommand{\U}{\mathcal{U}} % macro for realization
\newcommand{\V}{\mathcal{V}} % macro for realization
\newcommand{\FP}{\mathcal{FP}} % macro for fano plane code

\newcommand{\C}{\mathcal{C}}
\newcommand{\od}{\stackrel{\text{def}}{=}} % Define by equality
\newcommand{\mxl}[1]{\mathbf{#1}} % Define by equality

\begin{document}

\maketitle
\begin{abstract}
       We study the open and closed embedding dimensions of a convex 3-sparse code $\FP$, which records the intersection pattern of lines in the Fano plane. We show that the closed embedding dimension of $\FP$ is three, and the open embedding dimension is between four and six, providing the first example of a 3-sparse code with closed embedding dimension three and differing open and closed embedding dimensions. 
       We also investigate codes whose canonical form is quadratic, i.e. ``degree two" codes. 
       We show that such codes are realizable by axis-parallel boxes, generalizing a recent result of Zhou on inductively pierced codes. 

       We pose several open questions regarding sparse and low-degree codes.
       In particular, we conjecture that the open embedding dimension of certain 3-sparse codes derived from Steiner triple systems grows to infinity.

\end{abstract}
\section{Introduction}

%\paragraph{Codes and realizations.}
A \emph{(combinatorial) code} is any set system $\C\subseteq 2^{[n]}$. 
Elements of a code are called \emph{codewords}.
We typically abbreviate codewords by listing out their elements, e.g. $\{1,2,3\}$ is expressed more concisely as $123$. 
We also typically express inclusion-maximal codewords in boldface.

Given a collection $\U = \{U_1,\ldots, U_n\}$ of sets in $\R^d$, we can use a code to record how the sets intersect and cover one another:
\[
\code(\U) \od \big\{\sigma\subseteq [n]\  \big |\  \text{there exists $p\in \R^d$ such that $p\in U_i$ if and only if $i\in\sigma$}\big\}.
\]
In other words, we label every point $p\in \R^d$ according to which $U_i$ contain it, then collect all such labels to form $\code(\U)$.
The collection $\U$ is said to \emph{realize} a code $\C$ when $\code(\U)=\C$.
We are particularly interested in studying codes that are \emph{convex}, meaning that they can be realized by a collection of convex subsets of Euclidean space. 
For example, the code $\C = \{\mxl{124},\mxl{13}, \mxl{234}, 12, 23, 24, 1,2,3,4, \emptyset\}$ is convex and has a realization in $\R^2$, as shown in \Cref{fig:first-example}.

\begin{figure}[h]
\[
\includegraphics{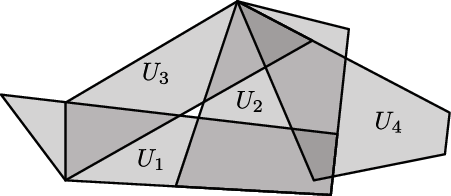}
\]
\caption{A planar convex realization of $\C = \{\mxl{124},\mxl{13}, \mxl{234}, 12, 23, 24, 1,2,3,4, \emptyset\}$. }\label{fig:first-example}
\end{figure}

Convex codes were introduced by Curto, Itskov, Veliz-Cuba, and Youngs~\cite{CIVY} to mathematically model hippocampal place cells. 
In this applied context, we are interested in \emph{open convex} codes---meaning that sets in a realization should be both convex and open---since the regions observed in experimental work are full-dimensional.
One can analogously define \emph{closed convex} codes, and perhaps surprisingly these two classes of codes differ.
Lienkaemper, Shiu, and Woodstock~\cite{LSW} described a code that is closed convex but not open convex, and Cruz, Giusti, Itskov, and Kronholm~\cite{CGIK} gave an example with the opposite behavior.
On the other hand, Franke and Muthiah~\cite{franke-muthiah} showed that every code be realized by convex sets in a large enough dimension when no topological requirements are placed on the sets.
The disparity between open and closed realizations motivated the introduction and study of \emph{open and closed embedding dimensions} of a code $\C\subseteq 2^{[n]}$, defined respectively as \begin{align*}
\odim(\C) &\od \min \{d \mid \text{$\C$ has an open convex realization in $\R^d$}\}, \text{\,\,\, and}\\
\cdim(\C) &\od \min \{d \mid \text{$\C$ has a closed convex realization in $\R^d$}\}.
\end{align*}
Above, the minimum over the empty set is defined to be $\infty$. 

Recent work of Jeffs~\cite{embedding-vectors} shows that there can be arbitrary differences between the open and closed embedding dimensions of a code.
In particular, for any $2\le a,b\le \infty$, there exists a code $\C$ with $\odim(\C) = a$ and $\cdim(\C) = b$. 
We are interested in whether or not such behavior remains present when we restrict to ``simple" codes. 
This paper investigates two distinct notions of being ``simple"---codes with low sparsity, and codes with low degree. 
Both notions are introduced below.

\paragraph{Sparse codes and $\FP$.}
We say that $\C\subseteq 2^{[n]}$ is \emph{$k$-sparse} if every codeword in $\C$ has cardinality at most $k$.
For example, the code realized in \Cref{fig:first-example} is 3-sparse. 
Jeffs, Omar, Suaysom, Wachtel, and Youngs~\cite{sparse} investigated 2-sparse codes, in particular showing that if $\C$ is open or closed convex and 2-sparse, then $\cdim(\C) = \odim(\C)\le 3$.
On the other hand, there are 3-sparse convex codes with $\odim(\C) \neq \cdim(\C)$, the first example being a code $\mathcal S_3$ described in work of Jeffs~\cite{embedding-phenomena} but implicit in earlier work of Lienkaemper, Shiu, and Woodstock~\cite{LSW}.

We are interested in the open and closed embedding dimensions of the 3-sparse \emph{Fano plane code}, \[
\FP \od \{\mxl{123},\mxl{145},\mxl{167},\mxl{246},\mxl{257},\mxl{347},\mxl{356}, 1,2,3,4,5,6,7,\emptyset\}.
\]
This is precisely the code obtained by recording the intersection pattern of lines in the Fano plane. 
Alternatively, this code arises from the unique Steiner triple system on seven indices by adding all singleton codewords and the empty codeword.
See \Cref{sec:conclusion} for discussion on this latter perspective.

For now, observe that $\FP$ is \emph{intersection complete}, meaning that the intersection of any two codewords is again a codeword.
This fact, together with the observations that $\FP$ is 3-sparse and has seven maximal codewords, and results of Cruz, Giusti, Itskov, and Kronholm~\cite{CGIK} and Jeffs~\cite{embedding-phenomena} imply that $\cdim(\FP) \le 5$ and $\odim(\FP)\le 6$.
In fact, we can refine these bounds significantly.

\begin{restatable}{theorem}{fpthm}\label{thm:fp}
The open and closed embedding dimensions of $\FP$ satisfy\[
\cdim(\FP) = 3\quad\quad\text{and}\quad\quad 4\le\odim(\FP)\le 6. 
\]
\end{restatable}

\noindent \Cref{thm:fp} provides the first example of a 3-sparse code with $\cdim(\C) = 3$ and $\odim(\C) > \cdim(\C)$. 
The general behavior of embedding dimensions for 3-sparse codes remains a wide open question.
In particular, it is not known if there can be a gap of more than one between the open and closed embedding dimensions of 3-sparse codes, nor whether the closed dimension can exceed the open dimension.
Perhaps most crucially, it remains unclear whether or not there is any uniform upper bound on the open or closed embedding dimensions of 3-sparse convex codes.
These questions and potential avenues of progress will be further discussed in \Cref{sec:conclusion}. 

\paragraph{Receptive field relations, the canonical form, and the degree of a code.}
Curto, Itskov, Veliz-Cuba, and Youngs~\cite{CIVY} took an algebraic approach to the study of codes and their realizations, generalizing the theory of Stanley--Reisner rings to arbitrary set systems. 
Their approach allows one to isolate minimal set-theoretic relationships that sets in a realization must satisfy. 
Below, we give a combinatorial and geometric account of their approach, which is equivalent to their algebraic framework. 

Given a code $\C\subseteq 2^{[n]}$ with a (not necessarily open, closed, or convex) realization $\U = \{U_1,\ldots, U_n\}$ in $\R^d$, we say that a pair $(\sigma,\tau)$ with $\sigma,\tau\subseteq [n]$ is a \emph{receptive field relation} or \emph{RF relation} if \[
\bigcap_{i\in\sigma} U_i\, \,\subseteq\,\, \bigcup_{j\in\tau} U_j.
\]
As usual, the empty union is the empty set, and we adopt the convention that $\bigcap_{i\in\emptyset} U_i$ is all of $\R^d$. 
We will only consider codes realizable by bounded sets, and so we never have $(\emptyset, \tau)$ as an RF relation.
Also, we note that the containment above is only interesting when $\sigma$ and $\tau$ are disjoint, and from here on we only work with RF relations where $\sigma\cap \tau = \emptyset$.
The RF relations of $\C$ do not depend on the realization $\U$, since $\C$ fully encodes the intersection and covering information of any of its realizations. 

We say that an RF relation $(\sigma,\tau)$ for $\C$ is \emph{minimal} if $(\sigma\setminus \{i\}, \tau)$ and $(\sigma, \tau\setminus \{j\})$ are not RF relations for any $i\in\sigma$ or $j\in\tau$. 
The \emph{canonical form} of a code $\C$ is \[
\CF(\C) \od \{(\sigma, \tau)\mid (\sigma,\tau) \text{ is a minimal RF relation for $\C$}\}. 
\]
The canonical form exactly captures the minimal set-theoretic (i.e. intersection and covering) relationships between sets in any realization of $\C$, and has been studied extensively from an algebraic perspective~\cite{CIVY, what-makes, grobner, polarization, signatures, canonical22}.
We say that the \emph{degree} of an RF relation $(\sigma,\tau)$ is $|\sigma| + |\tau|$, and the degree of a code $\C$ is the maximum degree of the relations in $\CF(\C)$. 

Curry, Jeffs, Youngs, and Zhao~\cite{CJYZ} showed that ``inductively pierced" codes have degree two, and can be realized not just by convex sets, but by open balls. 
In a recent master's thesis, Zhou~\cite{zhouthesis} showed that inductively pierced codes can also be realized by axis-parallel boxes.
These results suggest a relationship between the degree of a code, and the complexity of its possible geometric realizations: codes with low degree should have simpler realizations.
We add to the evidence of this trend for degree two codes with the following theorem, proved in \Cref{sec:boxes}.

\begin{restatable}{theorem}{degreetwo}\label{thm:degreetwo}
    Let $\C\subseteq 2^{[n]}$ be a degree two code. Then $\C$ can be realized by axis-parallel boxes in dimension $\max\{1, n-1\}$.
\end{restatable}

\section{Background and supporting lemmas.}
Before proving our results, we first recall some useful geometric and combinatorial results, and justify several supporting lemmas. 
Throughout, we will use the notation $\overline{pq}$ to denote the line segment between points $p$ and $q$ in $\R^d$.
We start by noting that a line spanned by a vertex of a simplex and one of its interior points must pass through the facet opposite the vertex in question. 
We set this apart as a lemma since we make use of it in several cases, but omit the proof. 

\begin{lemma}\label{pierce}
    Let $P=\{p_1,\ldots,p_{k+1}\}\subseteq \R^d$ with $k\leq d$ such that its points are in general position, and let $q$ be in the relative interior of the $k$-simplex $\conv(P)$. 
    Then the line $L$ which passes through $p_i$ and $q$ intersects the facet $\conv(P\setminus\{p_i\})$. 
\end{lemma}

\paragraph{Radon partitions.}
Given a set $P\subseteq \R^d$ with $|P| = d+2$, Radon's theorem guarantees that there exists a partition $P = P_1\sqcup P_2$ so that $\conv(P_1)\cap \conv(P_2)\neq\emptyset$.
In fact, when $P$ is in general position such a partition is unique, and $\conv(P_1)\cap \conv(P_2)$ consists of a single point, which we call the \emph{Radon point} of $P$. 

We will be interested in various cases based on the sizes of $P_1$ and $P_2$. 
To this end, the Radon partition $P_1\sqcup P_2$ is called an \emph{$n$--$m$ split} when $|P_1|=n$ and $|P_2|=m$ (we will assume $n\ge m$ throughout).
Given enough points in general position, we can always  find a split among a subset of them that is as even as possible. 
The following lemma explains such a situation in $\R^3$. 

\begin{lemma} \label{3--2}
    Let $P=\{p_1,\ldots,p_6\}\subseteq \R^3$ be in general position. Then there exists a 5-element subset of $P$ whose Radon partition is a 3--2 split.
\end{lemma}
\begin{proof}
    Consider the first five points. If these form a 3--2 split, we are done. Otherwise, they form a 4--1 split. Without loss of generality, let $p_5\in \conv(p_1,\ldots,p_4)$.
    We consider two cases.\medskip
    
    \noindent \textbf{Case 1:} $p_6\notin\conv(p_1,\ldots,p_4)$.\smallskip
    
    \noindent The line segment $\overline{p_5p_6}$ intersects a facet of the 3-simplex $\conv(p_1,\ldots,p_4)$. Without loss of generality, let this facet be $\conv(p_1,p_2,p_3)$. Then $\{p_1,p_2,p_3\}\sqcup\{p_5,p_6\}$ is a 3--2 Radon partition.\medskip
    
    \noindent\textbf{Case 2: }$p_6\in\conv(p_1,\ldots,p_4)$. \smallskip
    
    \noindent We may consider $\conv(p_1,\ldots, p_4)$ as the union of four smaller simplices, each with $p_5$ as a vertex, and the other three vertices coming from $\{p_1, p_2, p_3, p_4\}$. 
    The point $p_6$ lies in one of these simplices, say without loss of generality $p_6\in\conv(p_1,p_2,p_3,p_5)$. Then the line segment $\overline{p_4p_6}$ intersects a facet of this 3-simplex, and as in the first case we obtain a 3--2 Radon partition.
\end{proof}

\paragraph{Symmetries in the Fano plane.}
The Fano plane has a great deal of symmetry.
For our purposes, the most important fact is that the symmetry group of the Fano plane is ``doubly transitive." 
In our language, this means that any pair of maximal codewords can be mapped to any other pair of maximal codewords by a symmetry.
We record this in a lemma below, which we will make extensive use of to reduce the casework in our arguments.

\begin{lemma}\label{symmetry}
    Let $m_1,\ldots,m_7$ be the maximal codewords of $\FP$.
    Then for every choice of $(m_i,m_j)$, $(m_k,m_\ell)$ where $i\neq j$, $k\neq\ell$, there exists a permutation $\Pi$ of $[7]$ which maps $\Pi(m_i)=m_k$ and $\Pi(m_j)=m_\ell$, and under which $\FP$ is invariant.
\end{lemma}

\paragraph{Sunflowers of convex open sets.}
We say that $\U = \{U_1,U_2,\ldots, U_n\}$ is a \emph{sunflower} if \[
\code(\U ) = \{[n], 1, 2,\ldots, n,\emptyset\},\] i.e. if all the sets have a common intersection, and each set only appears alone outside this intersection.
Sunflowers of convex open sets have played an important role in the study of convex codes, in particular serving as building blocks to describe rich families of codes with gaps between open and closed embedding dimensions~\cite{sunflowers, embedding-phenomena}. 
The first implementation of these ideas was given by Lienkaemper, Shiu, and Woodstock~\cite{LSW} who proved the following result. 

\begin{lemma}[Lemma 3.2 of \cite{LSW}]\label{lem:LSW}
Let $\{U_1, U_2, U_3\}$ be a sunflower of convex open sets in $\R^d$.
If a line $L$ intersects each $U_i$, then in fact $L$ intersects $U_1\cap U_2\cap U_3$. 
\end{lemma}
\noindent The following lemma is an immediate consequence of this result, and will be used extensively in our analysis of the open embedding dimension of $\FP$. 
\begin{lemma}\label{collinear}
    Let $\{U_1,U_2,U_3\}$ be a sunflower of convex open sets in $\R^d$.
    Let $p_1\in U_1,p_2\in U_2,$ and $p_3\in U_3$ be points such that $p_3 \in \overline{p_1p_2}$.
    Then $p_3\in U_1\cap U_2\cap U_3$.
\end{lemma}
\begin{proof}
    \Cref{lem:LSW} guarantees that $\overline{p_1p_2}$ contains a point $p_{123}$ in $U_1\cap U_2\cap U_3$ 
    We then have $p_3\in \overline{p_1 p_{123}}$ or $p_3\in \overline{p_{123}p_2}$.
    The former case implies $p_3\in U_1$ by convexity of $U_1$, and similarly the latter case implies $p_3\in U_2$.
    Since $\{U_1, U_2, U_3\}$ is a sunflower, both cases in fact imply that $p_3$ lies in $U_1\cap U_2\cap U_3$, as desired. 
\end{proof}

A realization of $\FP$ contains seven different sunflowers, one for each maximal codeword.
In fact, $\FP$ contains many induced copies of the previously mentioned code $\mathcal S_3 = \{\mxl{123},\mxl{14},\mxl{24},\mxl{34}, 1,2,3,4,\emptyset\}$ implicit in Lienkaemper, Shiu, and Woodstock's work~\cite{LSW}. 
One can use \Cref{lem:LSW} to prove that $\cdim(\mathcal S_3) = 2$ and $\odim(\mathcal S_3) = 3$ (dimension-minimal realizations are shown in \Cref{fig:S3}).
The Fano plane code $\FP$ contains 28 different isomorphic copies of $\mathcal S_3$: one may take any of the seven maximal codewords, and add any of the remaining four indices to find such a copy. 
Our \Cref{thm:fp} can be thought of as saying that these 28 copies of $\mathcal S_3$ are ``sufficiently entangled" in $\FP$ that the closed dimension increases by one, and the open dimension increases by at least one.

\begin{figure} [h]
\[
\includegraphics{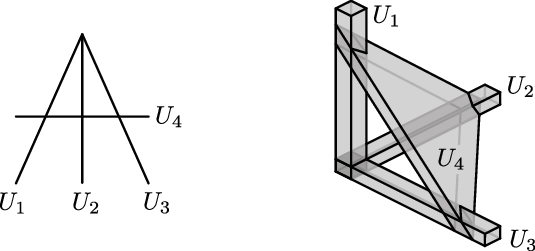} 
\]
\caption{Dimension-minimal closed and open realizations of $\mathcal S_3$ in $\R^2$ and $\R^3$.}\label{fig:S3}
\end{figure}

\section{The closed dimension of $\FP$}
\Cref{fig:fanoR3} shows a closed realization of $\FP$ in $\R^3$.
The sets in this realization are four facets of the octahedron, and three axis-parallel line segments passing through opposite vertices of the octahedron.
Formally, this realization is defined by \begin{align*}
X_1 &=\conv\{e_1, -e_1\} & X_4 &=\conv\{-e_1, -e_2, e_3\}\\
X_2 &= \conv\{e_2, -e_2\}& X_5 &=\conv\{-e_1, e_2, -e_3\} \\
X_3 &=\conv\{e_3, -e_3\} & X_6 &=\conv\{e_1, -e_2, -e_3\}\\
&& X_7 &=\conv\{e_1, e_2, e_3\}. 
\end{align*}

\noindent The remainder of this section is devoted to proving that the realization in \Cref{fig:fanoR3} is dimension-minimal. 
Throughout the proof below, we use the notation $L(p,q)$ to denote the line through points $p$ and $q$.

\begin{figure}[h]
\[
\includegraphics{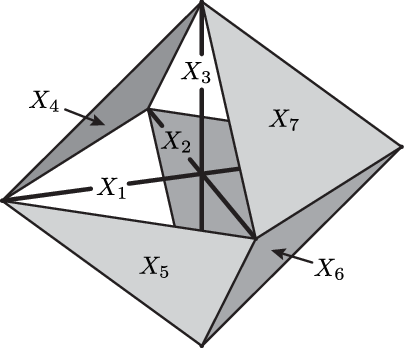}
\]
    \caption{A closed realization of $\FP$ in $\R^3$.}\label{fig:fanoR3}
\end{figure}

\begin{theorem}\label{thm:fanoR2}
    There is no closed convex realization of $\FP$ in the plane.
    Consequently, $\cdim(\FP) = 3$.
\end{theorem}
\begin{proof}
Suppose for contradiction that $\{X_1,\ldots, X_7\}$ is a closed convex realization $\FP$ in $\R^2$.
By intersecting our sets with a sufficiently large closed ball, we may assume that the realization is compact.
Choose a set of seven points 
\[
P = \{p_{ijk}\in X_i\cap X_j\cap X_k \mid \{i,j,k\}\in \FP\}.
\]
Observe that any such $P$ must have positive area.
If not, $P$ would be contained in a line $L$, and applying Helly's theorem to the collection of segments $L\cap X_i$ would yield a point in all seven sets, a contradiction.
By compactness, we may choose $P$ to minimize the area of $\conv(P)$: we are choosing seven points from disjoint compact subsets of $\R^2$, and the area of their convex hull varies continuously with the choices of points.
From here on, we assume that $P$ has minimum area among all possible choices of $P$, and consider two cases.
\smallskip

\noindent\textbf{Case 1: $\conv(P)$ is not a triangle.} 
It will suffice to restrict our attention to four vertices on the boundary of $\conv(P)$.
Consider the crossing point of the diagonals determined by these vertices.
Up to symmetry (i.e. by \Cref{symmetry}), we may assume that $p_{123}$ and $p_{145}$ are the endpoints of one of these diagonals.
Since the line segment between any two vertices is contained in some $X_i$, the crossing point is contained in two sets, and hence also in a third. 
The possible codewords arising at the crossing point are $123, 145$, and $167$, but the former two are impossible, for otherwise we could replace $p_{123}$ or $p_{145}$ by the crossing point, obtaining a smaller area for the convex hull of $P$.
Thus the codeword arising at the crossing point is $167$, and since this point lies in the convex hull of the other points, we can assume without loss of generality that it is in fact equal to $p_{167}$. 

Observe that $\FP$ is invariant under the permutation $(1)(2435)(67)$, and this permutation induces a cyclic permutation of the maximal codewords not containing 1, while transposing $123$ and $145$ and leaving $167$ fixed. 
Applying this permutation, we can assume that $p_{246}$ is one of our remaining two boundary points of interest, which leaves $p_{356}$ as the only valid choice for the opposite point.
The situation is shown in \Cref{fig:fanoR2-quad}. 
\begin{figure}[h]
\[
\includegraphics{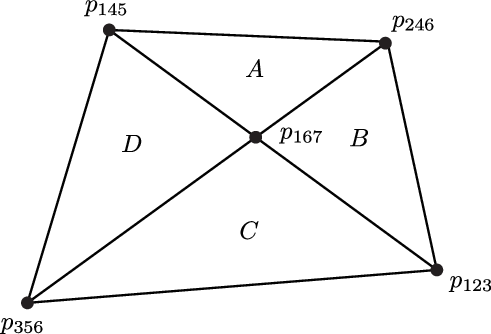}
\]
\caption{Case 1 in the proof of \Cref{thm:fanoR2}.}\label{fig:fanoR2-quad}
\end{figure}

Now, we claim that there is a point giving rise to one of the codewords $257$ or $347$ in the closed regions $A$, $B$, $C$, or $D$ in \Cref{fig:fanoR2-quad}. 
If $p_{257}$ lies in one of these regions the claim is immediate. 
If $p_{257}$ lies outside of the quadrilateral, then the segment $\overline{p_{167}p_{257}}$ is contained in $X_7$ and must cross one of the edges of the quadrilateral.
The edges lie in $X_2$, $X_3$, $X_4$, and $X_5$ respectively, so this crossing point must give rise to one of the codewords $257$ or $347$.

Considering the various cases, we see that we have have arrived at a contradiction.
If the codeword $257$ appears in region $A$ at a point $q$, then the crossing point of $\overline{q p_{123}}$ and $\overline{p_{167}p_{246}}$ lies in $X_2$, $X_4$, and $X_6$, and we could have replaced $p_{246}$ with this point to obtain a smaller convex hull.
Similar contradictions arise in regions $B$, $C$, and $D$. 
For example, in region $B$ the crossing point of the segments $\overline{qp_{145}}$ and $\overline{p_{167}p_{246}}$ will lie in $X_3$, $X_5$, and $X_6$, contradicting our choice of $p_{356}$. 
If the codeword $347$ arises in region $A$, $B$, $C$, or $D$ then examining crossing points of appropriate line segments yields analogous contradictions. 

\noindent\textbf{Case 2: $\conv(P)$ is a triangle.} 
Up to symmetry, we may assume that two of the vertices of this triangle are $p_{123}$ and $p_{145}$.
The third vertex cannot be $p_{167}$, since this would mean $P\subseteq X_1$, a contradiction.
We may assume without loss of generality that the third vertex is $p_{246}$ by applying the permutation $(1)(2435)(67)$ that we used in the previous case.
Applying an affine transformation, we can assume that $\conv(P)$ is an equilateral triangle with $p_{246}$ at its apex.

We now claim that any choice for the set $Q = \{p_{167}, p_{257}, p_{347}\}$ cannot be collinear.  
Suppose for contradiction that $Q$ can be chosen to lie on a line.
The permutation $(124)(365)(7)$ is a symmetry of $\FP$, and cyclically permutes the vertices of our equilateral triangle as well as the points in $Q$.
Applying this permutation, we may assume that $p_{167}$ lies between $p_{257}$ and $p_{347}$. 
But $p_{257}$ and $p_{347}$ must lie outside the triangle $\conv\{p_{123}, p_{145}, p_{167}\}\subseteq X_1$, and so one of the line segments $\overline{p_{257}p_{123}}$ or $\overline{p_{347}p_{123}}$ crosses the line segment $\overline{p_{145}p_{167}}$ (the first case is shown in \Cref{fig:fanoR2-1}). 
This crossing point is contained in $X_1$, $X_2$, and $X_3$, and could have been chosen as $p_{123}$ to yield a smaller area for $\conv(P)$, a contradiction. 
Hence $Q$ comprises the vertices of a triangle.
\begin{figure}[h]
\[
\includegraphics{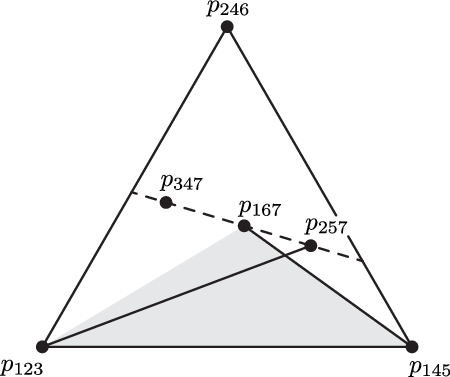}
\]
\caption{Proving that $Q = \{p_{167}, p_{257}, p_{347}\}$ cannot be collinear.}\label{fig:fanoR2-1}
\end{figure}

But now consider the point $p_{356}$.
This point cannot be contained in $\conv(Q)\subseteq X_7$.
Hence the line segment from $p_{356}$ to one of the vertices of $Q$ crosses the edge determined by the other two vertices of $Q$.
Again using the symmetry $(124)(365)(7)$, we can reduce to the situation shown in \Cref{fig:fanoR2-2}. 
But here the crossing point of $\overline{p_{167}p_{356}}$ and $\overline{p_{257}p_{347}}$ lies in $X_6$, $X_7$, and hence also $X_1$.
Thus we could have chosen $Q$ to be collinear, and we have arrived at a final contradiction.
\begin{figure}[h]
\[
\includegraphics{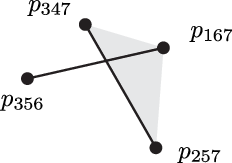}
\]
\caption{The final contradiction in Case 2: the set $Q$ could have been chosen to be collinear.}\label{fig:fanoR2-2}
\end{figure}
\end{proof}

\begin{remark}We have shown that $\FP$ cannot be realized by closed convex sets in $\R^2$. 
Our investigations strongly indicate that in fact $\FP$ cannot be realized by any convex sets (not necessarily open or closed) in $\R^2$, but it seems that a complete proof in the style of the one given above would entail significant casework.
We in fact conjecture the even stronger result that there is no collection of convex sets $\{C_1,\ldots, C_7\}$ with $C_i\cap C_j\cap C_k\neq\emptyset$ precisely when $ijk\in \FP$. 
In other words, we conjecture that the abstract simplicial complex $\Delta(\FP)$, which is obtained by adding every pair to $\FP$, is not ``$2$-representable."
For background on $d$-representable complexes, we recommend Tancer's 2011 survey paper~\cite{tancer-survey}.
\end{remark}
\section{The open dimension of $\FP$}

Cruz, Giusti, Itskov, and Kronholm~\cite{CGIK} showed that every intersection complete code with $m$ maximal codewords has open embedding dimension at most $\max\{2,m-1\}$, and hence it follows that $\odim(\FP)\le 6$.
All that remains to establish \Cref{thm:fp} is to prove that $\odim(\FP)>3$, which we will do below.
\begin{theorem}\label{odim3}
The Fano plane code has no open convex realization in $\R^3$. That is, $\odim(\FP)\ge 4$. 
\end{theorem}
\begin{proof}
    Suppose for the sake of contradiction that $\U=\{U_1,\ldots,U_7\}$ realizes $\FP$ with convex open sets in $\R^3$. 
    Choose $P=\{p_{123},p_{145},p_{167},p_{246},p_{257},p_{347},p_{356}\}$ such that $p_{ijk}\in U_i\cap U_j\cap U_k$ and the points are in general position. This can be done because $U_i\cap U_j\cap U_k$ are open and nonempty for the chosen points. 
    \par Employing \Cref{3--2}, we may choose a 3--2 Radon partition $P_1\sqcup P_2$ on five out of the first six points. Let $P_1$ be the set consisting of three points and $P_2$ be the set consisting of two points. 
    Let $q\in \conv(P_1)\cap \conv(P_2)$ be the Radon point of this partition.
    By \Cref{symmetry}, we can assume without loss of generality that $P_2=\{p_{123}, p_{145}\}$ (note using this symmetry may mean that $p_{356}$ ends up in $P_1$, however this will not cause any problems).
    By convexity of $U_1$, $\conv(P_2)\subseteq U_1$ so $q\in U_1$.
    We now consider two possible cases, both of which will establish that $q$ in fact lies in $U_1\cap U_6\cap U_7$.
    These cases are illustrated in \Cref{fig:R3-cases}.\medskip
    
    \noindent \textbf{Case 1.} All elements of $P_1$ are in a common set $U_6$ or $U_7$.\smallskip
    
    \noindent Without loss of generality, let $P_1=\{p_{167},p_{246},p_{356}\}\subseteq U_6$. A similar argument can be applied to $P_1=\{p_{167},p_{257},p_{347}\}\subset U_7$.
     By convexity of $U_6$, $\conv(P_1)\subseteq U_6$ so $q\in U_6$. Since $\U$ realizes $\FP$, the fact that $q$ lies in $U_1\cap U_6$ implies $q\in U_1\cap U_6\cap U_7$.\medskip
    
    \noindent  \textbf{Case 2.} All elements of $P_1$ are not in a common set.\smallskip
    
    \noindent Without loss of generality, let $P_1=\{p_{167},p_{246},p_{347}\}$. This can be assumed because the following argument relies on two points of $P_1$ being in a common set $U_6$ and the remaining point being in $U_7$ (or vice versa). Indeed, any 3 element subset of $\{p_{246},p_{356},p_{167},p_{257},p_{347}\}$ contains 2 elements from $U_6$ or from $U_7$ by the pigeonhole principle. By assumption, the remaining point will not be in a common set with these two and will therefore be in the opposing set ($U_7$ or $U_6$, respectively). 
    
    Now, consider the line $L$ through $p_{246}$ and $q$. By \Cref{pierce}, we may choose $q'\in L\,\cap\,\overline{p_{167}p_{347}}$. By convexity of $U_7$, $q' \in U_7$.
    Now, noting that $q\in \overline{p_{246},q^\prime}$, $p_{246}\in U_6$, $q\in U_1$, $q^\prime \in U_7$, and $\{U_1,U_6,U_7\}$ form a sunflower, we can apply \Cref{collinear} and conclude that $q\in U_1\cap U_6\cap U_7$.\medskip
    
     We have shown that in either case we have $q\in U_1\cap U_6\cap U_7$, and we are ready to derive our final contradiction.
     Since $q\in \overline{p_{123}p_{145}}$, $p_{123}\in U_2$, $p_{145}\in U_4$, and $q\in U_6$ and the collection $\{U_2,U_4,U_6\}$ forms a sunflower, we can apply \Cref{collinear} to conclude that $q\in U_2\cap U_4\cap U_6$.
     But then $q$ lies in $U_1\cap U_2\cap U_4\cap U_6\cap U_7$, which contradicts the assumption that $\U$ realizes $\FP$.
     \begin{figure}[h]
     \[
     \includegraphics{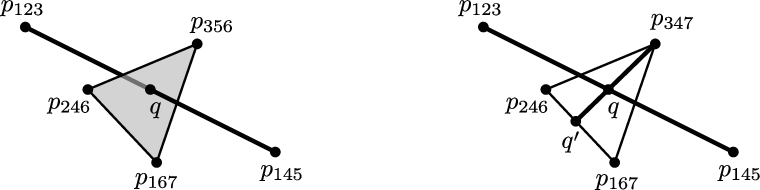}
     \]
     \caption{The two cases in the proof of \Cref{odim3}.}\label{fig:R3-cases}
     \end{figure}
\end{proof}

\begin{section}{Studying $\FP$ in dimension four.}
We are not sure whether or not $\odim(\FP)$ exceeds four.
This section will discuss the limitations of the argument used to prove \Cref{odim3} when applied in $\R^4$.
One may suppose for contradiction that  $\U = \{U_1,\ldots, U_7\}$ is an open convex realization of $\FP$ in $\R^4$, and similarly choose a point set $P = \{p_{123},p_{145},p_{167},p_{246},p_{257},p_{347},p_{356}\}$ with each point chosen from the intersection of sets it is labeled by. 
We still have sufficiently many points to apply Radon's theorem, but now we have three cases to consider: a 5--1 split, a 4--2 split, and a 3--3 split.

The first two cases are easier to consider, since the smaller part of the split will be contained in some $U_i$, and hence the Radon point will be contained in some $U_i$.
In fact, similar---albeit lengthier and more technical---geometric arguments to the ones used to prove \Cref{odim3} are sufficient to rule these cases out.

When a Radon partition of points in $P$ is a 3--3 split, the Radon point may or may not be contained in some $U_i$.
For example, if $P = \{p_{123}, p_{246}, p_{347}\}\sqcup\{p_{145}, p_{167}, p_{356}\}$, then neither part of the partition is contained in any $U_i$, and it is not clear how to proceed. 
However, since $P$ has seven points, there are seven different subsets of size six which we can try.
Could all of these lead to the ``bad" 3--3 split case?

The answer is in fact yes, with the relevant arrangement being given by points along the moment curve in $\R^4$.
In particular, suppose that the points of $P$ appear in the following order along the moment curve:\[
p_{123} <
p_{145} <
p_{246} <
p_{167}<
p_{347}<
p_{356}<
p_{257}.
\]
 The relevant property of this ordering is that all six pairs of adjacent points, as well as the pair of endpoints $p_{123}$ and $p_{257}$, are contained in a  unique set $U_i$.  Radon partitions of points along the moment curve are always 3--3 splits whose parts ``interlace," i.e. alternate with one another.
 For example, the Radon partition of the last six points is given by $P_1=\{p_{145}, p_{167}, p_{356}\}$ and $P_2=\{p_{246},p_{347},p_{257}\}$.
 One may check that all seven Radon partitions which follow this alternating pattern will yield a 3--3 split such that $P_1$ and $P_2$ are both not a subset of any $U_i$.
 Hence we are not able to derive a contradiction analogous to our previous arguments in this case. 

It is reasonable to speculate that such an arrangement could be used to construct an open convex realization of $\FP$ in $\R^4$. 
Starting with the seven triangles which are convex hulls of the points containing a given index, we have a closed convex realization of $\FP$.
Could this be ``thickened" appropriately to obtain an open realization, for example by taking a Minkowski sum with a carefully chosen open convex set?
This may be the most promising path towards determining the exact open embedding dimension of $\FP$. 

\end{section}

\section{Realizing degree two codes with boxes}\label{sec:boxes}

We will use and inductive approach to prove \Cref{thm:degreetwo}, facilitated by the lemmas below.
Before proceeding, we note several features of degree two codes that will be useful in our proofs.
First, there are only two types of RF relation with degree two: \begin{align*}
    &\text{$(\{i,j\}, \emptyset)$ which corresponds to $U_i\cap U_j = \emptyset$, and}\\
   &\text{$(\{i\}, \{j\})$ which corresponds to $U_i\subseteq U_j$.}
\end{align*}
We will write these relations more concisely as $(ij, \emptyset)$ and $(i, j)$ respectively.
We pause briefly to justify that deleting an index does not increase the degree of a code. 

\begin{lemma}\label{lem:delete}
Let $\C\subseteq 2^{[n]}$ be a degree two code.
For any $i\in[n]$, the code \[
\C\setminus i \od \{c\setminus \{i\}\mid c \in \C\}
\]
is degree at most two.
\end{lemma}
\begin{proof}
Let $(\sigma,\tau)$ be a minimal RF relation of $\C \setminus i$.
Every RF relation for $\C\setminus i$ is an RF relation for $\C$, and hence $(\sigma,\tau)$ is in fact a minimal RF relation for $\C$.
Hence $(\sigma,\tau)$ has degree at most two, and it follows that $\C\setminus i$ is degree at most two. 
\end{proof}

The only RF relations with degree one are of the form $(i,\emptyset)$, corresponding to $U_i = \emptyset$.
By relabeling such indices to the end, and then forgetting about them, we can always reduce a degree two code $\C$ to an equivalent code where every RF relation has degree \emph{exactly} two, and in particular all sets in a realization of $\C$ are nonempty. 

We say that an index $i\in[n]$ is \emph{inclusion minimal} in a code $\C$ if there is no $j\neq i$ with $(j, i)$ an RF relation of $\C$ (equivalently, if $U_i$ is inclusion-minimal among all sets in any realization of $\C$). 
Note that inclusion minimal indices always exist, provided that $\C$ does not have any indices which appear in identical sets of codewords.
If two indices $i$ and $j$ do appear in identical sets of codewords, then we must have $U_i = U_j$ in every realization of $\C$, and thus we must have $(i,j)$ and $(j,i)$ both as RF relations for $\C$. 
If we are only interested in forming a realization of $\C$ of a certain type (by convex sets, or boxes, for example) then we can apply \Cref{lem:delete} to delete one of these indices.
In this way, we can reduce to the case where every two indices in $\C$ have distinct behavior. 

One last important feature of degree two codes is that they are \emph{intersection complete}: the intersection of any two codewords is again a codeword.
This fact is nontrivial to prove, but can be inferred as an immediate consequence of work of Curto, Gross, Jeffries, Morrison, Rosen, Shiu, and Youngs~\cite[Proposition 3.7]{signatures}.
This fact streamlines the proof of the following lemma.

\begin{lemma}\label{lem:deletion-is-subset}
Let $\C\subseteq 2^{[n]}$ be a degree two code, and let $i\in[n]$ be an inclusion minimal index.
Then $\C\setminus i$ is a subset of $\C$.
\end{lemma}
\begin{proof}
Let $c\in \C$ be a codeword with $i\in c$.
We must argue that $c\setminus \{i\}$ is also a codeword of $\C$.
For contradiction, suppose not.
Observe that every codeword containing $c\setminus \{i\}$ must also contain $i$: otherwise we could intersect such a codeword with $c$ to obtain $c\setminus \{i\}$ as a codeword, since degree two codes are intersection complete.
This means that $(c\setminus \{i\}, i)$ is an RF relation.
Since $\C$ is degree two, this relation must reduce to a minimal relation of degree at most two.
A relation of the form $(\{j,k\},\emptyset)$ with $j,k\in c\setminus  \{i\}$ is not possible since $c$ is a codeword, and so we must have a relation $(j, i)$ where $j\in c\setminus \{i\}$.
This contradicts the fact that $i$ is inclusion minimal. 
\end{proof}

The most important tool for our proof is the following lemma, which allows us to extend a realization of $\C\setminus n$ in $\R^d$ to a realization of $\C$ in $\R^{d+1}$ whenever $\C$ is degree two and $n$ is inclusion-minimal.

\begin{lemma}\label{lem:extend}
Let $\C\subseteq 2^{[n]}$ be a degree two code, and suppose $n$ is an inclusion-minimal index. 
Define \begin{align*}
\sigma &= \{i\in [n-1] \mid \text{$(n, i)$ is an RF relation}\},\quad \text{and}\\
\tau &= \{i\in [n-1] \mid \text{$(\{i, n\},\emptyset)$ is not an RF relation}\}.
\end{align*}
Given a realization $\U = \{U_1,\ldots, U_{n-1}\}$ of $\C\setminus n$ in $\R^d$, the collection $\V = \{V_1,\ldots, V_n\}$ given by \[
V_i = \begin{cases}
U_i \times [0,1] & \text{if }i\in [n-1]\setminus \tau,\\
U_i\times [0,3] & \text{if }i \in \tau,\\
\left(\bigcap_{j\in\sigma} U_j\right) \times [2,3] & \text{if }i = n.
\end{cases}\]
is a realization of $\C$ in $\R^{d+1}$. 
\end{lemma}
\begin{proof}
%Fix $p\in \R^d$, we consider the case of $\pi_n(p)\in[0,1]$ and $\pi_n(p)\in[1,3]$ separately. If $\pi_n(p)\in[0,1]$, then $p$ give rise to codewords of $\mathcal{C}\setminus n$. By Lemma $\ref{lem:deletion-is-subset}$, these codewords are also in $\mathcal{C}$. If $\pi_n(p)\in[1,3]$, and suppose $p$ gives rise to a codeword not in $\mathcal{C}$. Then $p$ also gives rise to a codeword not in $\mathcal{C}\setminus n$. Let $(i,j)$ denote the reduction into degree two relation, with $i\in\tau$ and $j\notin \tau$. Then $j\cap n=\emptyset$, contradicts with the fact that $(i,j)$ is an RF while $(\{i,n\},\emptyset)$ is not an RF relation.
Fix $p\in \R^{d+1}$, and consider the codeword that arises at $p$ in the realization $\V$.
If the last coordinate of $p$ lies outside the range $[0,3]$, then we simply obtain the empty codeword at $p$.
Let $q$ denote the projection of $p$ to $\R^d$, i.e. the point obtained by setting the last coordinate of $p$ to zero. 
Let $c\in \C\setminus n$ denote the codeword that arises at $q$ in the realization $\U$. 
If the last coordinate of $p$ lies in the range $[0,1]$, then the codeword arising at $p$ in the realization $\V$ is exactly $c$. 
In particular, the codewords arising for such $p$ are exactly those in $\C\setminus n$, which is a subset of $\C$ by \Cref{lem:deletion-is-subset}. 

It remains to consider the case that the last coordinate of $p$ lies in the range $(1,3]$. 
Here we must carefully consider several cases.\medskip

\noindent\textbf{Case 1:} $p\notin V_n$. \smallskip

\noindent The codeword arising at $p$ in $\V$ will be precisely $c\cap \tau$. 
Let $\gamma$ denote $c\cap \tau$ and suppose for contradiction that $\gamma$ is not a codeword of $\C$, and in particular not a codeword of $\C\setminus n$.
This means that $(\gamma, \delta)$ is an RF relation of $\C\setminus n$ for some $\delta\subseteq [n-1]\setminus \tau$. 
Since $\C\setminus n$ is degree two and $c$ is a codeword of $\C\setminus n$ this reduces to an RF relation $(i, j)$ where $i\in \gamma$ and $j\in [n-1]\setminus \tau$.
The latter condition implies that $(\{j,n\}, \emptyset)$ is an RF relation in $\C$. 
But these two relations together imply that $(\{i, n\},\emptyset)$ is an RF relation in $\C$, contradicting the fact that $i\in \tau$.\medskip 

\noindent\textbf{Case 2:} $p\in V_n$. \smallskip
Observe that the codewords arising at such $p$ in $\V$ are precisely of the form $(c\cap \tau)\cup \{n\}$ where $c$ is a codeword of $\C\setminus n$ that contains $\sigma$. 
It thus suffices to show that these are precisely the codewords of $\C$ that contain $n$. 

First suppose that $\tilde c$ is a codeword of $\C$ containing $n$.
Then $\sigma\subseteq \tilde c$ because $\sigma$ by definition records the indices in $[n-1]$ that appear in every codeword of $\C$ containing $n$, and $\tilde c\setminus \{n\} \subseteq \tau$ since every index in $\tilde c\setminus \{n\}$ appears together with $n$ in the codeword $\tilde c$. 
Setting $c = \tilde c\setminus \{n\}$, we see that $c$ is a codeword of $\C\setminus n$ containing $\sigma$, and $\tilde c = (c\cap \tau)\cup \{n\}$ as desired. 

For the converse, let $c$ be a codeword of $\C\setminus n$ that contains $\sigma$.
Let $\gamma = c\cap \tau$, and note that the argument from Case 1 shows that $\gamma$ is a codeword of $\C\setminus n$. 
Suppose for contradiction that $\gamma \cup \{n\}$ is not a codeword of $\C$.
Then $(\gamma \cup\{n\}, \delta)$ is an RF relation for $\C$, for some $\delta\subseteq [n-1]\setminus \sigma$. 
This reduces to a degree two relation, but each possibility leads to a contradiction.
A relation $(\{i,j\},\emptyset)$ where $\{i,j\}\subseteq\gamma$ is not possible since $\gamma$ is a codeword.
A relation $(\{i,n\},\emptyset)$ with $i\in\gamma$ is not possible since $\gamma\subseteq \tau$.
A relation $(i, j)$ where $i\in \gamma$ and $j\in \delta$ is not possible since $(\gamma, \delta)$ is not a relation. 
Finally, a relation $(n, i)$ where $i\in\delta$ is not possible since $\delta$ is disjoint from $\sigma$.

We have thus shown that the codewords arising inside $V_n$ in $\V$ are exactly the codewords of $\C$ that contain $n$, concluding the proof. 
\end{proof}

\begin{example}\label{ex:degree-two}
Consider the code \begin{align*}
\C &= \{\mxl{123}, \mxl{1345}, 135, 145, 134, 12, 13, 1, 4, \emptyset\}.
\end{align*}
The minimal RF relations for this code are \[
(\{2,4\},\emptyset),\,\,   (\{2,5\},\,\,   \emptyset),\,\,   (2,1),\,\,   (3,1),\,\,   (5,1),\,\,   (5,3).
\]
In particular, it is a degree two code.
Moreover, $\C\setminus 5$ has a realization by intervals in $\R^1$.
\Cref{fig:degree-two} shows the construction of \Cref{lem:extend} applied to this realization, in order to obtain a realization of $\C$ by axis-parallel boxes in $\R^2$. 
The sets $\sigma$ and $\tau$ of \Cref{lem:extend} are $\sigma = \{3\}$ and $\tau = \{1,3,4\}$ in this case. 
\begin{figure}[h]
    \[
    \includegraphics{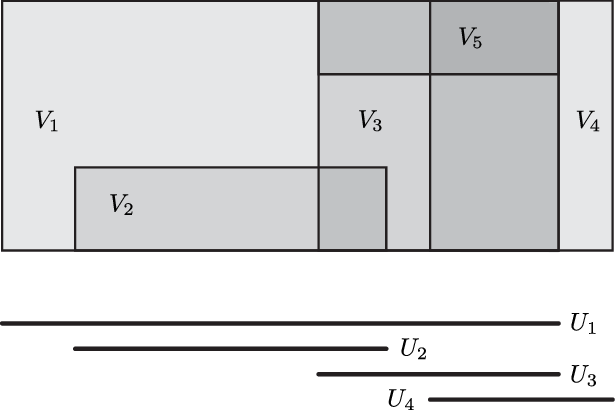}
    \]
    \caption{The construction of \Cref{lem:extend}, extending a realization of $\C\setminus 5$ in $\R^1$ (below) to a realization of $\C$ in $\R^2$ (above). }\label{fig:degree-two}
\end{figure}

\end{example}

\degreetwo*
\begin{proof}
We proceed by induction on $n$. 
The base cases $n=1$ and $n=2$ can be verified straightforwardly, since every code on one or two indices is degree two and also realizable by intervals in $\R^1$. 
For the inductive step with $n\ge 3$, fix a degree two code $\mathcal{C}\subseteq 2^{[n]}$. By \Cref{lem:delete}, $\mathcal{C}\setminus n$ is also degree two, and by inductive hypothesis there exists a realization $\{U_1,\ldots, U_{n-1}\}$ of $\mathcal{C}\setminus n$ in $\R^{n-2}$ by axis-parallel boxes.
The realization of $\C$ in $\R^{n-1}$ provided by \Cref{lem:extend} consists of products of the various $U_i$ and their intersections with intervals, and hence also consists of axis-parallel boxes. 
\end{proof}

\section{Conclusion}\label{sec:conclusion}

Several lines of investigation remain open. 
Perhaps the most pressing question is to resolve the ambiguity regarding the open embedding dimension of $\FP$.
\begin{question}
What is the precise value of $\odim(\FP)$?
\end{question}

Our study of $\FP$ was motivated by the broader question of studying 3-sparse codes.
A more general family of 3-sparse codes can be obtained from ``Steiner triple systems," which are sets of triples in $[n]$ where every pair in $n$ appears in a unique triple in the system.
A Steiner triple system on $n$ exists precisely when $n\equiv 1$ or $n\equiv 3$ modulo 6, and the maximal codewords of $\FP$ form the smallest nontrivial Steiner triple system, which is in fact the unique such system on seven indices.
Given any Steiner triple system, one can form an associated convex code by adding the singletons and the empty codeword.
Call such a code a \emph{Steiner triple code}.
Every Steiner triple code is 3-sparse and intersection complete, and so has closed embedding dimension at most five (see \cite[Theorem 1.9]{embedding-phenomena}). 
However, our \Cref{thm:fp} shows that the open embedding dimension can exceed the closed embedding dimension in a Steiner triple code.
As we have seen, a reason for this is that realizations of such codes contain many sunflowers of three sets, to which we can potentially apply \Cref{collinear}. 
We posit that as the Steiner systems in question grow, so must the open embedding dimension.
\begin{conjecture}
    For every $d\ge 1$ there exists a Steiner triple code $\C$ with $\odim(\C)\ge d$. 
\end{conjecture}

The route to establishing this conjecture is not at all straightforward. 
Our proof that $\FP$ is not open convex in $\R^3$ made frequent use of the Fano plane's symmetries, and also the property that any two maximal codewords share a unique index.
This poses a challenge to generalizing our methods to higher order Steiner triple codes, and new techniques may be needed. 

As regards codes of low degree, a natural next step is to investigate degree three codes.
Another interesting question is to study codes which are both sparse \emph{and} low degree. 
\begin{question}
Among convex degree three codes, what pairs of embedding dimensions can arise? 
\end{question}
\begin{question}
Can we determine  bounds on the embedding dimensions of $k$-sparse, degree $\ell$ codes, in terms of $k$ and $\ell$? 
\end{question}

\bibliographystyle{plain} 
\bibliography{main.bib}

\end{document}